\documentclass{amsart}
\usepackage{amscd,amssymb,latexsym}
\usepackage{tikz}
\usepackage{tikz-3dplot}

\begin{document}
\def\a{{\mathbf a}}
\def\b{{\beta}} 	\def\ra{{\rangle}}	\def\PP{{\mathbb P}}  	\def\ZZ{{\mathbb Z}}  
\def\u{{\mathbf u}}\def\v{{\mathbf v}} \def\w{{\mathbf w}}
\def\bb{{\mathbf b}}

\def\reg{\operatorname {reg}}		   
\def\supp{\operatorname {supp}} 
\def\Tor{\operatorname {Tor}} 	

\newtheorem{Theorem}{Theorem}[section]
\newtheorem{Corollary}[Theorem]{Corollary}
\newtheorem{Lemma}[Theorem]{Lemma}
\newtheorem{Proposition}[Theorem]{Proposition}
\newtheorem{Question}[Theorem]{Question}
\newtheorem{Example}[Theorem]{Example}
\newtheorem{Definition}[Theorem]{Definition}
\newtheorem{Remark}[Theorem]{Remark}
\newtheorem{Conjecture}{Conjecture}
\numberwithin{equation}{section}

\title{Periodicity of betti numbers of monomial curves}  
\author{Thanh Vu}
\address{Department of Mathematics, University of California at Berkeley, Berkeley, CA 94720}
\email{vqthanh@math.berkeley.edu}
\date{\today}

\begin{abstract} Let $K$ be an arbitrary field. Let $\a = (a_1< \cdots <a_n)$ be a sequence of positive integers. Let $C(\a)$ be the affine monomial curve in ${\mathbb A}^n$ parametrized by $t\to (t^{a_1}, ..., t^{a_n})$. Let $I(\a)$ be the defining ideal of $C(\a)$ in $K[x_1, ..., x_n]$. For each positive integer $j$, let $\a+j$ be the sequence $(a_1 + j, ..., a_n+j)$. In this paper, we prove the conjecture of Herzog and Srinivasan saying that the betti numbers of $I(\a + j)$ are eventually periodic in $j$ with period $a_n -a_1$. When $j$ is large enough, we describe the betti table for the closure of $C(\a+j)$ in $\PP^n$. 
\end{abstract}
\maketitle

\section{Introduction}
Let $K$ denote an arbitrary field. Let $R$ be the polynomial ring $K[x_1,...,x_n]$. Let $\a = (a_1< \cdots<a_n)$ be a sequence of positive integers. The sequence $\a$ gives rise to a monomial curve $C(\a)$ whose parametrization is given by $x_1 = t^{a_1}, ..., x_n = t^{a_n}.$ Let $I(\a)$ be the defining ideal of $C(\a)$. For each positive integer $j$, let $\a+j$ be the sequence $(a_1 + j, ..., a_n + j)$. In this paper, we consider the behaviour of the betti numbers of the defining ideals $I(\a+j)$ and their homogenizations $\bar I(\a+j)$ for positive integers $j$. 

For each finitely generated $R$-module $M$ and each integer $i$, let 
$$\b_i(M) = \dim_K \Tor_i^R(M,K)$$
be the $i$-th total betti number of $M$. The following conjecture was communicated to us by Herzog and Srinivasan. 

\begin{Conjecture}[Herzog-Srinivasan]\label{Conj} The betti numbers of $I(\a+j)$ are eventually periodic in $j$ with period $a_n - a_1$. 
\end{Conjecture}
In general, the problem of finding defining equations of monomial curves is difficult. For example, in \cite{B}, Bresinsky gave an example of a family of monomial curves in ${\mathbb A}^4$ whose numbers of minimal generators of the defining ideals are unbounded. Recently, in the case $n\le 3$, Conjecture \ref{Conj} was proven by Jayanthan and Srinivasan in \cite{JS}. In the case when $\a$ is an arithmetic sequence, Conjecture \ref{Conj} was proven by Gimenez, Sengupta and Srinivasan in \cite{GSS}. In this paper, we prove the conjecture in full generality:

\begin{Theorem}\label{main} The betti numbers of $I(\a+j)$ are eventually periodic in $j$ with period $a_n - a_1.$
\end{Theorem}

To prove Theorem \ref{main} we first prove the eventual periodicity in $j$ for total betti numbers of the homogenization $\bar I(\a+j)$, and then prove the equalities for total betti numbers of $I(\a+j)$ and $\bar I(\a+j)$ when $j \gg 0$. 

To simplify notation, for each $i$, $1 \le i \le n$, let $b_i = a_n - a_i$. Note that if $f$ is homogeneous then $f \in I(\a)$ if and only if $f \in I(\a+j)$ for all $j$. Denote by $J(\a)$ the ideal generated by homogeneous elements of $I(\a)$. In general, for each finitely generated graded $R$-module $M$, $\Tor_i^R(M,K)$ is a finitely generated graded module for each $i$. Let 
$$\b_{ij}(M) = \dim_K \Tor_i^R(M,K)_j$$
be the $i$-th graded betti number of $M$ in degree $j$. Moreover, let
$$\reg M = \sup_{i,j}\{j - i: \b_{ij} \neq 0\}$$
be the Castelnuovo-Mumford regularity of $M$. 

In Lemma \ref{separation}, we prove that when $j > b_1 (n + \reg J(\a))$, each binomial in $\bar I(\a + j)$ involving $x_0$ has degree greater than $n + \reg J(\a)$. Thus the betti table of $\bar I(\a + j)$ separates into two parts. One part is the betti table of $J(\a)$ which lies in degree at most $n + \reg J(\a)$. The other part lies in degree larger than $n + \reg J(\a)$. We call the part of betti table of $\bar I(\a+ j)$ lying in degree larger than $n + \reg J(\a)$ the high degree part. We will prove that when $j \gg 0$, the betti table of $\bar I(\a + j + b_1)$ is obtained from the betti table of $\bar I(\a+j)$ by shifting the high degree part of $\bar I(\a + j)$ by certain rows (see Theorem \ref{Per}). 

\begin{Example}\label{Ex1} Let $\a = (1,2,3,7,10)$. A computation in Macaulay2 shows that the betti tables of $\bar I(\a+49)$ and $\bar I(\a+58)$ are as follows:
\begin{center}
\begin{tabular}{c | c c c c c c c c |c c c c c}
	& $0$ 		&$1$		&$2$ 		&$3$	&	&	& 	& 	& $0$ 		&$1$		&$2$ 		&$3$		\\ 	
\cline{1-5} \cline{9-13}
$2$ 	& $1$		&$-$ 		&$-$		&$-$	&	&	&	& $2$ 	& $1$		&$-$ 		&$-$		&$-$	\\
$3$ 	& $6$ 		&$8$		&$1$ 		&$-$	&	&	&	& $3$ 	& $6$ 		&$8$		&$1$ 		&$-$\\
$4$ 	&$-$ 		&$2$ 		&$4$ 		&$1$	&	&	&	& $4$ 	&$-$ 		&$2$ 		&$4$ 		&$1$\\
$5$ 	&$-$		&$-$ 		&$-$		&$-$	&	&	&	& $5$ 	&$-$		&$-$ 		&$-$		&$-$\\
$6$ 	&$-$ 		&$-$ 		&$-$ 		&$-$	&	&	&	& $6$ 	&$-$ 		&$-$ 		&$-$ 		&$-$\\
$7$ 	&$2$ 		&$1$ 		&$-$ 		&$-$	&	&	&	& $7$ 	&$-$ 		&$-$ 		&$-$ 		&$-$\\
$8$ 	&$1$		&$11$ 	&$13$		&$3$	&	&	&	& $8$ 	&$2$ 		&$1$ 		&$-$ 		&$-$\\
$9$ 	&$-$ 		&$-$ 		&$-$ 		&$1$	&	&	&	& $9$ 	&$1$		&$11$ 	&$13$		&$3$\\
 	& 		& 		& 		&	&	&	&	& $10$ 	&$-$ 		&$-$ 		&$-$ 		&$1$\\
\end{tabular}
\end{center}
where the entry in the column-index $i$ and row-index $j$ of each table represents the betti number $\b_{i,i+j}$ of the corresponding ideals. A dash represents $0$.
\end{Example}

To prove the shifting behaviour of the betti tables of $\bar I(\a + j)$ as well as the equalities of total betti numbers of $I(\a+j)$ and $\bar I(\a+j)$, we note that $\bar I(\a+j)$ and $I(\a+j)$ are defining ideals of certain semigroup rings. Moreover, by \cite[Proposition~1.1]{BH}, betti numbers of a semigroup ring can be given in term of homology groups of certain simplicial complexes associated to elements of the semigroup. Thus we reduce the problem to proving equalities among homology groups of these simplicial complexes.

Let $V$ be an additive semigroup generated by vectors $\v_1, ..., \v_n \in {\mathbb N}^m$. Let $K[V]$ be the semigroup ring generated by $V$. Let $I(V)$ be the defining ideal of $K[V]$ in $R = K[x_1, ..., x_n]$. 
\begin{Definition}[Squarefree divisor simplicial complex]\label{divisorcomplex} For each $\v\in V$, let $\Delta_\v$ be the simplicial complex on the vertices $\{1, ..., n\}$ such that $F \subseteq \{1, ..., n\}$ is a face of $\Delta_\v$ if and only if 
$$\v - \sum_{i\in F} \v_{i}\in V.$$
\end{Definition}

Note that $K[x_1, ..., x_n]$ is multi-graded with grading given by $\deg x_i = \v_i$. Under this grading, by \cite[Proposition 1.1]{BH}, (see also \cite[Theorem 1.2]{CM}), the betti numbers of $I(V)$ and the homology groups of $\Delta_\v$ are related by 
\begin{Theorem}\label{trans} For each $i$, and each element $\v \in V$, 
$$\b_{i,\v}(I(V)) = \b_{i+1,\v}(R/I(V)) = \dim_K \tilde H_i (\Delta_\v).$$
\end{Theorem}
In our situation, betti numbers of $\bar I(\a+j)$ can be expressed in terms of homology groups of squarefree divisor simplicial complexes $\Delta_{l,r}$ (defined in section \ref{doublecone}). In section \ref{doublecone}, we prove that if $l > n + \reg J(\a)$ and $\Delta_{l,r}$ has non-trivial homology groups, then $\Delta_{l,r}$ is a double cone. By a double cone, we mean the union of two cones. From that, we derive the equalities of homology groups among these squarefree divisor simplicial complexes. The following example illustrates the double cone structure on $\Delta_{l,r}$. 

\begin{Example} Let $\a = (1,2,3,7,10)$. We consider the betti numbers of $\bar I(\a+49)$. The betti table in Example \ref{Ex1} shows that $\b_{1,9}(\bar I(\a+49))$ and $\b_{2,10}(\bar I(\a + 49))$ are nonzero. A more precise computation in Macaulay2 with multi-grading shows that $\Delta_{9,73}$ and $\Delta_{10,83}$ contribute to the betti numbers $\b_{1,9}$ and $\b_{2,10}$ of $\bar I(\a + 49)$ respectively. The complex $\Delta_{9,73}$ is the simplicial complex on the vertices $\{0, ..., 5\}$ with facets $0524, 053, 124, 134$, which is the double cone $\{05\}* \{3,24\}\cup \{1\}*\{24,34\}$. Also, $\Delta_{10,83}$ is the simplicial complex on the vertices $\{0, ..., 5\}$ with facets $0514, 05123, 0534, 124, 134,$ which is the double cone $\{05\}*\{14,123,34\} \cup \{1\} *\{24,34\}.$ The picture for these simplicial complexes are given in the following where we have identified $0$ and $5$.
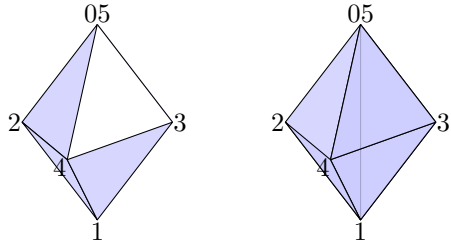
\begin{center}
\begin{tikzpicture}[join=round]
    \tikzstyle{conefill} = [fill=blue!20,fill opacity=0.8]
    \tikzstyle{ghostfill} = [fill=white]	
    \filldraw[conefill](2,1.3)--(1.6,-.5)--(1,0)--cycle;
    \filldraw[conefill](1.6,-.5)--(2,-1.3)--(1,0)--cycle;
    \filldraw[conefill](1.6,-.5)--(2,-1.3)--(3,0)--cycle;
    \filldraw[conefill](2,1.3)--(3,0);
    \node at (2,-1.45) {$1$};
    \node at (0.9,0) {$2$};
    \node at (1.5,-.6) {$4$};
    \node at (2,1.45) {$05$};
    \node at (3.1,0) {$3$};
    \begin{scope}[xshift=3.5cm]
    \filldraw[conefill](2,1.3)--(3,0)--(2,-1.3)--(1,0)--cycle;
    \filldraw[conefill](2,1.3)--(1.6,-.5)--(2,-1.3)--cycle;
    \filldraw[conefill](3,0)--(2,-1.3)--(1.6,-0.5)--cycle;
    \filldraw[conefill](1,0)--(2,-1.3)--(1.6,-0.5)--cycle;
    \filldraw[conefill](2,1.3)--(3,0)--(1.6,-0.5)--cycle;
%    \filldraw[ghostfill](1,0)--(2,1.3)--(1.6,-0.5)--cycle;
    \node at (2,-1.45) {$1$};
    \node at (0.9,0) {$2$};
    \node at (1.5,-.6) {$4$};
    \node at (2,1.45) {$05$};
    \node at (3.1,0) {$3$};
    \end{scope}
\end{tikzpicture}
\end{center}
\end{Example}

Note that the separation of betti tables of $\bar I(\a+j)$ happens as long as $j > b_1 (n+ \reg J(\a))$. It is natural to expect that the periodicity of the betti table of $\bar I(\a + j)$ begins when $j > b_1 (n+\reg J(\a))$. Experiments suggest that this is correct. Nevertheless our current method will only give a slightly larger bound for the place when the periodicity happens as follows.

Let $d = \gcd(b_1, ...,b_{n-1})$ be the greastest common divisor of $b_1, ..., b_{n-1}$. Let $c$ be the conductor of the numerical semigroup generated by $b_1/d, ..., b_n/d$ (see \cite{RS} for more details). Let $B = \sum_{i=1}^n b_i + n + d$. Let 
\begin{equation}\label{Eq1}
N = \max \left \{b_1 (n+\reg J(\a)), b_1b_2\left  ( \frac{dc +b_1}{b_{n-1}} + B\right ) \right \}.
\end{equation}

Fix $j > N$. Let $k = a_n + j$. Let $e = d/\gcd(d,k)$. In the case $l > n + \reg J(\a)$, we prove that, for each pair $(l,r)$ whose $\Delta_{l,r}(j)$ has non-trivial homology groups then $\Delta_{l,r}(j) = \Delta_{l+e,r+eb_1}(j+b_1)$ proving that the betti numbers of $\bar I(\a+j)$ and $\bar I(\a+j+b_1)$ are equal (see section \ref{proj} for more details).

Denote by $(\a+j)$ the semigroup generated by $a_1+j, ..., a_n+j$. For each pair $(l,r)$ corresponding to an element of $\overline{\a+j}$, $m = lk -r$ is an element of $(\a+j)$. In the case $l > n + \reg J(\a)$, we prove that if $\Delta_{l,r}$ has non-trivial homology groups then $\Delta_m$ is obtained from $\Delta_{l,r}$ by the deletion of the vertex $0$. The double cone structures on $\Delta_{l,r}$ and $\Delta_m$ show that they have the same homology groups, proving the equalities of betti numbers of $I(\a+j)$ and $\bar I(\a+j)$ (see section \ref{general} for more details). Consequently, we prove that the betti numbers of $I(\a+j)$ are periodic in $j$ with period $b_1$ when $j > N$. The technical condition \eqref{Eq1} will naturally arise in the proofs of lemmas throughout the paper. 

Finally, we consider the behaviour of betti numbers of $I(\a + j)$ in the case $\a$ is a Bresinsky's sequence. In this case, we prove in Proposition \ref{Bre} that the period $b_1$ is exact. 

The paper is organized as follows. In section \ref{notation}, we recall various notation in the introduction which will be used throughout the paper. In section \ref{doublecone}, we prove the double cone structure on squarefree divisor simplicial complexes $\Delta_{l,r}$ when $l > n + \reg J(\a)$ and $\Delta_{l,r}$ has non-trivial homology groups. In section \ref{proj}, applying the double cone structure, we prove equalities among squarefree divisor simplicial complexes associated to $\bar I(\a + j)$ and $\bar I(\a + j + b_1)$. As a corollary, we have a description of the betti table of $\bar I(\a + j)$ when $j \gg 0$. Finally, in section \ref{general}, applying the double cone structure, we prove the relation between squarefree divisor simplicial complexes associated to $I(\a + j)$ and $\bar I(\a + j)$. As a consequence, we prove our main theorem. Finally, we prove that the period estimated in the main theorem is sharp for the examples considered by Bresinsky in \cite{B}.

%%%%%%%%%%%%%%%%%%%%
%%%%%%%%%%%%%%%%%%%%%5
%%%%%%%%%%%%%%%%%%%%%%%
\section{Preliminaries}\label{notation}
The following notation and facts will be used throughout the paper. 
\begin{itemize}
\item For each $i=1, ..., n$, let $b_i = a_n-a_i$. In particular, $b_n = 0$. \\
\item Let $d$ be the greatest common divisor of $b_1, ..., b_{n-1}$. \\
\item Let $c$ be the conductor of the semigroup generated by $b_1/d, ..., b_{n-1}/d$.\\
\item Let $B =  \sum_{i=1}^n b_i + n +d$. \\
\item Let $N = \max \left \{b_1 (n+\reg J(\a)), b_1b_2\left  ( \frac{dc +b_1}{b_{n-1}} + B\right ) \right \}.$\\
\item Fix $j > N$. Let $k = a_n + j$. Let $e = d/\gcd(d,k)$. Note that $d|b_1$, so $e = d/\gcd(d,k+b_1)$. 
\end{itemize}

Since $k > N$, it follows that $k$ satisfies
\begin{equation}\label{reg} 
k/b_1 > n+\reg J(\a),
\end{equation}
and 
\begin{equation}\label{Eq2.1}
\frac{k}{b_1b_2} > \frac{dc + b_1}{b_{n-1}} + B.
\end{equation}

We will use the following notation when dealing with faces of simplicial complexes on the vertices $\{0, ..., n\}$. If $F\subseteq \{0, ..., n\}$, let $|F|$ be the cardinality of $F$. Moreover, for $i \in \{0, ..., n\}$, let
$$\delta_{iF} = \begin{cases} 0 & \text{ if } i\notin F,\\
  1 & \text{ if } i\in F.\end{cases}$$

Finally, the following representation of a natural number will be used frequently in the paper. Let $u$ be a natural number such that $u$ is divisible by $d$ and $u \ge dc$. We can write $u = tb_1 + v$ for some natural number $t$ and $v$ such that $dc \le v < dc + b_1$. Because $d|b_1$, it follows that $d|v$. Since $v/d \ge c$, the conductor of the numerical semigroup generated by $b_1/d, ..., b_{n-1}/d$, we can write 
$$v/d = w_1 (b_1/d) + ... + w_{n-1}(b_{n-1}/d),$$
for non-negative integers $w_1, ..., w_{n-1}$. If we denote by $\bb = (b_1, ..., b_{n-1})^t$ and $\w = (w_1, ..., w_{n-1})^t$ the column vectors with coordinates $b_1, ..., b_{n-1}$ and $w_1, ..., w_{n-1}$ respectively, then we have the following representation of $u=tb_1 + v$ as 
\begin{equation}\label{rep}
u = t b_1 + \w \cdot \bb,
\end{equation}
where $\w \cdot \bb = \sum_{i=1}^{n-1} w_i b_i$ is the usual dot product of these two vectors. With this representation, if we denote by $|\w| = \sum_{i=1}^{n-1} w_i$ then
\begin{equation}\label{twbound}
t < \frac{u}{b_1} \text{ and } |\w| < \frac{dc+b_1}{b_{n-1}}.
\end{equation}

%%%%%%%%%%%%%%%%%%
%%%%%%%%%%%%%%%55
%%%%%%%%%%%%%%%%%%5
\section{Double cone structure on simplicial complexes $\Delta_{l,r}$}\label{doublecone}
In this section, we prove the double cone structure of the squarefree divisor simplicial complexes $\Delta_{l,r}$ defined below. 

Note that $\bar I(\a+j)$ is the defining ideal of the semigroup ring $K[\overline{\a+j}]$ where $\overline{\a+j}$ is the additive semigroup generated by vectors $\w_0 = (k,0),  \w_1 = (b_1,k-b_1), ..., \w_n = (0,k)$. Note that $|\w_i| = k$ for all $i = 0, ..., n$. Thus each element $\v$ of the semigroup $\overline{\a+j}$ corresponds uniquely to a pair $(l,r)$ where $l = |\v|/k$ and $r$ is the first coordinate of $\v$. The definition of squarefree divisor simplicial complex in Definition \ref{divisorcomplex} translates to 
\begin{Definition}\label{deltalr}
For each pair of natural numbers $(l,r)$, let $\Delta_{l,r}(j)$ be the simplicial complex on the vertices $\{0,...,n\}$ such that $F \subseteq \{0,...,n\}$ is a face of $\Delta_{l,r}(j)$ if and only if the equation 
$$y_0 k + \sum_{i=1}^n y_i b_i = r$$
has a non-negative integer solution $y = (y_0,..., y_n)$ such that $|y| = \sum_{i=0}^n y_i = l$ and $F \subseteq \supp y =\{i: y_i > 0\}$.
\end{Definition}

By Theorem \ref{trans}, if we consider $R[x_0]$ with standard grading $\deg x_i = 1$ then the graded betti numbers of $\bar I(\a + j)$ are given by
\begin{Proposition}\label{barj} For each $i$ and each $l$, we have
$$\b_{il}(\bar I(\a+j)) = \dim_K \Tor^{R[x_0]}_i(\bar I(\a+j),K)_l = \sum_{r \ge 0} \dim_K \tilde H_i (\Delta_{l,r}(j)).$$
\end{Proposition}

An easy consequence of inequality \eqref{reg} and Proposition \ref{barj} is the separation of the betti table of $\bar I(\a + j)$ when $j > N$.
\begin{Lemma}\label{separation} Assume that $j > N$. Any minimal binomial generator of $\bar I(\a + j)$ involving $x_0$ has degree greater than $n + \reg J(\a)$. In particular, any syzygy of $\bar I(\a + j)$ of degree at most $n + \reg J(\a)$ is a syzygy of $J(\a)$.
\end{Lemma}
\begin{proof} Assume that $f = x_0^{u_0} x_1^{u_1}...x_n^{u_n} - x_1^{v_1}...x_n^{v_n}$ is a minimal binomial generator of $\bar I(\a+j)$. By Definition \ref{deltalr} and Proposition \ref{barj}, we have 
$$u_0k + u_1b_1 + ... + u_nb_n = v_1 b_1 + ... + v_nb_n.$$
Since $u_0 > 0$, it follows that $v_1 b_1 +...+v_nb_n \ge k$. By inequality \eqref{reg},
$$\deg f = \sum_{i=1}^n v_i \ge k/b_1 > n + \reg J(\a).$$
The second part follows immediately.
\end{proof}

In this section, we will simply denote by $\Delta_{l,r}$ the simplicial complex $\Delta_{l,r}(j)$. We will prove that in the case $l >n+ \reg J(\a)$ and $\Delta_{l,r}$ has non-trivial homology groups, $\Delta_{l,r}$ is a double cone, the union of a cone over the vertices $\{0,n\}$ and another cone over the vertex $1$. 

Our first technical lemma says that the range for $r$ so that $\Delta_{l,r}$ has non-trivial homology groups is quite small.

\begin{Lemma}\label{Lem2.1} For any $l >n+ \reg J(\a)$, if $\Delta_{l,r}$ has non-trivial homology groups then $ek \le r < ek + dc +B$ and $l \ge r/b_1$. In particular, any solution $y = (y_0, ..., y_n)$ of the equation $y_0 k + y_1b_1+... + y_n b_n = r$ with $y_0 > 0$ satisfies $y_0 = e$.
\end{Lemma}
\begin{proof} Since $l > n + \reg J(\a)$, if $\Delta_{l,r}$ has non-trivial homology groups then it supports a nonzero syzygies of $\bar I(\a+j)$ in degree larger than $n + \reg J(\a)$. By Lemma \ref{separation}, $\Delta_{l,r}$ must have at least a facet containing $0$. By Definition \ref{deltalr}, the equation 
\begin{equation}\label{EL1.1}
y_0 k + y_1b_1 + ... + y_nb_n=r
\end{equation}
has a non-negative integer solution $y = (y_0, ..., y_n)$ such that $y_0 > 0$. Moreover, for $\Delta_{l,r}$ to have non-trivial homology groups, it must have at least a facet that does not contain $\{0\}$. Again, by Definition \ref{deltalr}, the equation \eqref{EL1.1} has a solution $z=(z_0, ..., z_n)$ such that $z_0 = 0$ and $\sum_{i=1}^n z_i = l$. This implies that $d |r$ and $l \ge r/b_1$. Therefore, we must have $d | y_0 k$ or $e|y_0$. Thus, $r \ge y_0 k \ge ek$. 

Now assume that $r \ge ek + dc + B$. Since $r - ek -\sum_{i=1}^{n-1} b_i > dc$, as in \eqref{rep}, we can write 
\begin{equation}\label{EL1.2}
r = ek + \sum_{i=1}^{n-1}b_i + tb_1 + \w \cdot \bb
\end{equation}
with 
\begin{equation}\label{EL1.3}
|\w| < \frac{dc+b_1}{b_{n-1}} \text{ and } t < \frac{r - ek}{b_1}.
\end{equation}

By inequality \eqref{Eq2.1}, it follows that
$$\frac{dc+b_1}{b_{n-1}} + B + \frac{r - ek}{b_1} < \frac{k}{b_1b_2} + \frac{r - ek}{b_1} \le \frac{r}{b_1} \le l.$$
Together with \eqref{EL1.3}, this implies
$$|\w| + e + n + t  < l.$$

Therefore, the equation \eqref{EL1.1} has a solution $u = (u_0, ..., u_n)$ such that $u_0 = e$, $u_1 = t + 1 + w_1$, $u_i = w_i +1$ for $i = 2, ..., n-1$ and $u_n = 1 +l - (t+e+n+|\w|)$. In particular, $u_i > 0$ for all $i$; consequently, by Definition \ref{deltalr}, $\Delta_{l,r}$ is the simplex $\{0, ..., n\}$ which has trivial homology groups. This is a contradiction. 

Finally, for any solution $y = (y_0, ..., y_n)$ of the equation \eqref{EL1.1} with $y_0 > 0$, we have 
$$y_0 \le \frac{r}{k} < \frac{ek + dc+B}{k} < e + 1.$$
Since $e|y_0$, it follows that $y_0 = e$.
\end{proof}

The following two lemmas will prove that if $l > n + \reg J(\a)$ and $\Delta_{l,r}$ has non-trivial homology groups then $\Delta_{l,r}$ has a structure of a double cone.

\begin{Lemma}\label{Lem2.2} Assume that $l >n+\reg J(\a)$ and $\Delta_{l,r}$ has non-trivial homology groups. If $F$ is a facet of $\Delta_{l,r}$, and $0\notin F$ then $1\in F$. 
\end{Lemma}
\begin{proof} Assume that there exists a facet $F$ of $\Delta_{l,r}$ such that $0, 1\notin F$. By Definition \ref{deltalr}, the equation 
\begin{equation}\label{EL2.1}
y_1b_1 + ... + y_n b_n= r
\end{equation}
has a solution $y  = (y_1, ...,y_n)$ such that $\supp y = F$, $y_1 = 0$ and $\sum_{i=1}^n y_i = l.$ Thus $l\ge \frac{r}{b_2}$. As in \eqref{rep}, we can write 
$$r = \sum_{i \in F} b_i + tb_1+ \w \cdot \bb$$
with 
$$|\w| < \frac{dc+b_1}{b_{n-1}} \text{ and } t < \frac{r}{b_1}.$$
Together with inequality \eqref{Eq2.1}, this implies
$$|\w| +n + t < \frac{dc+b_1}{b_{n-1}} + B + \frac{r}{b_1} < \frac{k}{b_1b_2} + \frac{r}{b_1} \le \frac{r}{b_1b_2} + \frac{r}{b_1} \le \frac{r}{b_2} \le l.$$
Note that $t > 0$, since $b_1 + ... + b_{n-1} + dc + b_1 < k \le r$. Therefore, there is a solution $z = (z_1, ..., z_n)$ of the equation \eqref{EL2.1} such that $z_1 = t+w_1, z_i = \delta_{iF} + w_i$ for $i = 2, ..., n-1$, and $z_n = l - (|\w| + t + |F|)$. In particular, $\supp z \supsetneq \supp y$ which is a contradiction.
\end{proof}

\begin{Lemma}\label{Lem2.3} Assume that $l > n + \reg J(\a)$ and $\Delta_{l,r}$ has non-trivial homology groups. If $F$ is a facet of $\Delta_{l,r}$, and $0\in F$ then $n\in F$. 
\end{Lemma}
\begin{proof} Assume that $F$ is a facet of $\Delta_{l,r}$ such that $0\in F$ and $n \notin F$. By Definition \ref{deltalr} and Lemma \ref{Lem2.1}, the equation 
$$ek + y_1b_1+... + y_nb_n = r$$
has a solution $y = (y_1, ..., y_n)$ such that $y_n = 0$ and $\sum_{i=1}^{n-1} y_i = l-e$. Moreover, by Lemma \ref{Lem2.1}, we have $l \ge \frac{r}{b_1}$. By inequality \eqref{Eq2.1}, it follows that
$$\left ( \frac{r}{b_1} - e \right ) b_{n-1} > \left ( \frac{k}{b_1}-e \right )b_{n-1} > \left ( \left ( \frac{dc+b_1}{b_{n-1}}+B \right )b_2-e \right ) b_{n-1} \ge dc+B.$$
Therefore, 
$$r = ek  + \sum_{i=1}^n y_i b_i \ge ek + \left ( \frac{r}{b_1} - e \right ) b_{n-1} > ek + dc + B,$$
which is a contradiction to Lemma \ref{Lem2.1}. Therefore, $n \in F$.
\end{proof}

One of the surprising consequences of the double cone structure is the following characterization of minimal inhomogeneous generators of $I(\a+j)$ when $j \gg 0$.

\begin{Corollary}\label{gens} Assume that $j > N$. Let $e = d/\gcd(d,a_n+j)$. Any minimal binomial inhomogeneous generator of $I(\a + j)$ is of the form 
$$x_1^u f - gx_n^v$$
where $u,v > 0$, $f$ and $g$ are monomials in the variables $x_2, ..., x_{n-1}$, and moreover, $u + \deg f = v + \deg g + e.$ 
\end{Corollary}
\begin{proof} By Lemma \ref{Lem2.1} and the double cone structure, any minimal binomial homogeneous generator of $\bar I(\a +j)$ involving an $x_0$ has the form
$$x_1^u f - x_0^e g x_n^v$$
where $u,v > 0$ and $f, g$ are monomials in $x_2, ..., x_{n-1}$. The corollary follows since any minimal binomial inhomogeneous generator of $I(\a + j)$ is obtained from dehomogenization of a binomial homogeneous generator of $\bar I(\a + j)$ involving an $x_0$.
\end{proof}

%%%%%%%%%%%%%%%%%%%%%
%%%%%%%%%%%%%%%%%%%%%
\section{Periodicity of betti numbers of projective monomial curves}\label{proj}

The main result of this section is Theorem \ref{Per} where we prove that when $j > N$, the betti table of $\bar I(\a + j+b_1)$ is obtained from the betti table of $\bar I(\a + j)$ by shifting the high degree part by $e$ rows.  As in section \ref{doublecone}, we denote by $\Delta_{l,r}(j)$ the squarefree divisor simplicial complexes associated to elements of the semigroup $\overline{\a + j}$ and $\Delta_{l,r}(j+b_1)$ the squarefree divisor simplicial complexes associated to elements of the semigroup $\overline{\a+j+b_1}$. 

As an appliciation of the double cone structure, we will prove that if $l > n + \reg J(\a)$ and $\Delta_{l,r}$ has non-trivial homology groups, then $\Delta_{l,r}(j) = \Delta_{l+e,r+eb_1}(j+b_1)$ as simplicial complexes. First we prove that if $l > n + \reg J(\a)$ and $\Delta_{l,r}(j)$ has non-trivial homology groups then $l$ is controlled in a small range by $r$.

\begin{Lemma}\label{Lem3.1} If $l > n + \reg J(\a)$ and $\Delta_{l,r}(j)$ has non-trivial homology groups, then 
$$l < \frac{r}{b_1} + \frac{dc +b_1}{b_{n-1}} + n.$$
\end{Lemma}
\begin{proof} In order for $\Delta_{l,r}(j)$ to have non-trivial homology groups, there must exist at least a facet $F$ of $\Delta_{l,r}(j)$ such that $n\notin F$. By Lemma \ref{Lem2.3}, this implies that $0\notin F$. Therefore, by Definition \ref{deltalr}, the equation
\begin{equation}\label{EL3.1}
y_1 b_1 + ... + y_n b_n = r
\end{equation}
has a solution $y = (y_1, ...,y_n)$ such that $\supp y = F$, $y_n = 0$, $\sum_{i=1}^{n-1} y_i = l$. As in \eqref{rep}, we can write 
$$r = \sum_{i\in F} b_i + tb_1 + \w\cdot \bb,$$
with 
$$|\w| < \frac{dc+b_1}{b_{n-1}} \text{ and } t < \frac{r}{b_1}.$$
Therefore, if $l \ge \frac{r}{b_1} + \frac{dc +b_1}{b_{n-1}} + n$ then $t + |\w| + n < l$. In particular, the equation \eqref{EL3.1} has a solution $z_1 = t + w_1$, $z_i = \delta_{iF} +w_i$ for $2\le i \le n-1$ and $z_n = l - t - |\w| - |F| > 0$. By Definition \ref{deltalr}, $F \cup \{n\} \subseteq \supp z$ is a face of $\Delta_{l,r}(j)$, which is a contradiction.
\end{proof}

The following two lemmas are the first indication of a relation between $\Delta_{l,r}(j)$ and $\Delta_{l+e,r+eb_1}(j+b_1)$.

\begin{Lemma}\label{Lem3.4} If $l > n + \reg J(\a)$ and $\Delta_{l,r}(j)$ has non-trivial homology groups, then 
$$\Delta_{l,r}(j) \subseteq \Delta_{l+e,r+eb_1}(j+b_1).$$
\end{Lemma}
\begin{proof} Let $F$ be a facet of $\Delta_{l,r}(j)$. We need to prove that $F \in \Delta_{l+e,r+eb_1}(j+b_1)$. If $0\notin F$ then by Definition \ref{deltalr}, the equation 
$$y_1 b_1 + ... + y_nb_n = r$$
has a solution $y = (y_1, ..., y_n)$ such that $\supp y =  F$ and $\sum_{i=1}^n y_i = l$. Therefore, the equation 
$$y_1b_1 + ... + y_nb_n= r+ eb_1$$
has a solution $z=(y_1+e, y_2, ..., y_n)$ such that $\supp z \supseteq  F$ and $\sum_{i=1}^n z_i = l+e.$ By Definition \ref{deltalr}, $F \in \Delta_{l+e,r+eb_1}(j+b_1).$

If $0 \in F$, then by Definition \ref{deltalr} and Lemma \ref{Lem2.1}, the equation 
$$y_0k + y_1 b_1 + ... +y_nb_n = r$$
has a solution $y = (e,y_1, ..., y_n)$ such that $\supp y = F$ and $\sum_{i=0}^n y_i = l$. Thus the equation 
$$y_0(k+b_1) + y_1 b_1 + ... + y_nb_n = r +eb_1$$
has a solution $z = (e,y_1, ..., y_{n-1}, y_n+e)$ such that $\supp z \supseteq F$ and $\sum_{i=0}^n z_i = l+ e$. By Definition \ref{deltalr}, $F \in \Delta_{l+e,r+eb_1}(j+b_1).$
\end{proof}

Similarly, we have 
\begin{Lemma}\label{Lem3.3}If $l > n + \reg J(\a)$ and $\Delta_{l,r}(j+b_1)$ has non-trivial homology groups, then
$$\Delta_{l-e,r-eb_1}(j) \subseteq \Delta_{l,r}(j+b_1).$$
\end{Lemma}
\begin{proof}Let $F$ be a facet of $\Delta_{l-e,r-eb_1}(j)$. We need to prove that $F \in \Delta_{l,r}(j+b_1)$. If $0\notin F$ then by Definition \ref{deltalr}, the equation 
$$y_1 b_1 + ... + y_nb_n = r-eb_1$$
has a solution $y = (y_1, ..., y_n)$ such that $\supp y =  F$ and $\sum_{i=1}^n y_i = l-e$. Therefore, the equation 
$$y_1b_1 + ... + y_nb_n= r$$
has a solution $z=(y_1+e, y_2, ..., y_n)$ such that $\supp z \supseteq  F$ and $\sum_{i=1}^n z_i = l.$ By Definition \ref{deltalr}, $F \in \Delta_{l,r}(j+b_1).$

If $0 \in F$ then by Definition \ref{deltalr}, the equation 
$$y_0k + y_1 b_1 + ... +y_nb_n = r-eb_1$$
has a solution $y = (y_0, ..., y_n)$ such that $\supp y = F$ and $\sum_{i=0}^n y_i = l-e$. Since $\Delta_{l,r}(j+b_1)$ has non-trivial homology group, the proof of Lemma \ref{Lem2.1} shows that $d|r$, thus $e|y_0$. Also, by Lemma \ref{Lem2.1}, $r < e(k+b_1)+dc+B$, thus $y_0 = e$. Therefore, the equation 
$$y_0(k+b_1) + y_1 b_1 + ... + y_nb_n = r$$
has a solution $z = (e,y_1, ..., y_{n-1}, y_n+e)$ such that $\supp z \supseteq F$ and $\sum_{i=0}^n z_i = l$. By Definition \ref{deltalr}, $F \in \Delta_{l,r}(j+b_1).$
\end{proof}

\begin{Proposition}\label{EquHom} If $l > n + \reg J(\a)$ and $\Delta_{l,r}(j)$ has non-trivial homology groups, then 
$$\Delta_{l,r}(j) = \Delta_{l+e,r+eb_1}(j+b_1).$$
\end{Proposition}
\begin{proof} By Lemma \ref{Lem3.4}, it suffices to show that for any facet $F$ of $\Delta_{l+e,r+eb_1}(j+b_1)$, we have $F \in \Delta_{l,r}(j)$. 

If $0 \notin F$, then by Definition \ref{deltalr}, the equation 
$$y_1 b_1 + ... + y_nb_n = r + eb_1$$
has a solution $y = (y_1, ..., y_n)$ such that $\supp y =  F$ and $\sum_{i=1}^n y_i = l+e$. Assume that $y_1 \le e$, then 
$$r = \sum_{i=1}^n y_i b_i - eb_1 \le \sum_{i=2}^n y_i b_i <(l+e)b_2.$$
By Lemma \ref{Lem3.1} and inequality \eqref{Eq2.1},
$$r < (l+e)b_2 < \left ( \frac{r}{b_1} + \frac{dc+b_1}{b_{n-1}} + n + e\right ) b_2 < \left ( \frac{r}{b_1} + \frac{r}{b_1b_2} \right ) b_2 \le r.$$

Thus we must have $y_1 > e$. Therefore, the equation
$$y_1b_1 + ... + y_nb_n = r$$
has the solution $z = (y_1-e, ...,y_n)$ such that $\supp z = \supp y = F$ and $\sum_{i=1}^n z_i = l$. By Definition \ref{deltalr}, $F \in \Delta_{l,r}(j).$ 

If $0\in F$, then by Definition \ref{deltalr} and Lemma \ref{Lem2.1}, the equation 
$$e(k+b_1) + y_1 b_1 + ... + y_nb_n = r+eb_1$$
has a solution $y = (y_1, ..., y_n)$ such that $\{0\} \cup \supp y= F$ and $e + \sum_{i=1}^n y_i = l+e.$ Assume that $y_n \le e$, then $\sum_{i=1}^{n-1}y_i \ge l-e$. By Lemma \ref{Lem2.1}, and inequality \eqref{Eq2.1}, 
\begin{align*}
r - ek& = \sum_{i=1}^{n-1} y_i b_i \ge (l-e)b_{n-1} \ge \left ( \frac{r}{b_1} - e \right )b_{n-1}\\
 & >b_2(dc+b_1 + Bb_{n-1}) - eb_{n-1} > dc + B.
\end{align*}
In other words, $r> ek + dc + B$. By Lemma \ref{Lem2.1}, $\Delta_{l,r}(j)$ has trivial homology, which is a contradiction. Thus, $y_n > e$. In particular, the equation 
$$ek + y_1 b_1 + ... + y_n b_n = r$$
has the solution $z = (y_1, ..., y_{n-1}, y_n-e)$ such that $\supp z = \supp y = F$ and $e +\sum_{i = 1}^n z_i = l$. Therefore, by Definition \ref{deltalr}, $F \in \Delta_{l,r}(j).$
\end{proof}

\begin{Proposition}\label{EqH} If $l> n + \reg J(\a)$ and $\Delta_{l,r}(j+b_1)$ has non-trivial homology groups, then 
$$\Delta_{l,r}(j+b_1) = \Delta_{l-e,r-eb_1}(j).$$
\end{Proposition}
\begin{proof} The proof is similar to that of Proposition \ref{EquHom}.
\end{proof}

The equality of the simplicial complexes $\Delta_{l,r}(j)$ and $\Delta_{l+e,r+eb_1}(j+b_1)$ shows that the betti table of $\bar I(\a + j+b_1)$ is obtained from the betti table of $\bar I(\a + j)$ by shifing the high degree part by $e$ rows as follows.

\begin{Theorem}\label{Per} If $l \le n + \reg J(\a)$, then
$$\b_{i,l}(\bar I(\a + j) = \b_{i,l} (\bar I(\a + j+b_1).$$
If $l > n + \reg J(\a)$, then 
$$\b_{i,l}(\bar I(\a+j)) = \b_{i,l+e} (\bar I(\a + j+b_1)).$$
\end{Theorem}
\begin{proof} First part follows since syzygies of $\bar I(\a + j)$ and $\bar I(\a + j+b_1)$ of degrees at most $n + \reg J(\a)$  are the syzygies of $J(\a)$ by Lemma \ref{separation}. 

Assume that $l >n+ \reg J(\a)$. By Proposition \ref{EquHom}, if $\Delta_{l,r}(j)$ has non-trivial homology groups, then $\Delta_{l,r}(j) = \Delta_{l+e,r+eb_1}(j+b_1).$ Thus, by Proposition \ref{barj}, it follows that
\begin{align*}
\b_{i,l} (\bar I(\a+j)) &= \sum_{r\ge 0} \dim_K \tilde H_i (\Delta_{l,r}(j)) \\
& = \sum_{r \ge 0} \dim_K \tilde H_i (\Delta_{l+e,r+eb_1}(j+b_1))\\
&\le \b_{i,l+e}(\bar I(\a + j+b_1)).
\end{align*}

Moreover, by Proposition \ref{EqH}, if $\Delta_{l+e,r}(j+b_1)$ has non-trivial homology groups, then $\Delta_{l+e,r}(j+b_1) = \Delta_{l,r-eb_1}(j).$ Thus, by Proposition \ref{barj}, it follows that
\begin{align*}
\b_{i,l+e}(\bar I(\a+j+b_1)) & = \sum_{r \ge eb_1} \dim_K \tilde H_i(\Delta_{l,r}(j+b_1))\\
& = \sum_{r\ge eb_1} \dim_K \tilde H_i(\Delta_{l,r-eb_1}(j))\\
& \le \b_{i,l} (\bar I(\a+j)).
\end{align*}

Therefore, $\b_{i,l}(\bar I(\a+j)) = \b_{i,l+e}(\bar I(\a+j+b_1)).$
\end{proof}

\begin{Remark}\label{Rem} Note that to establish results in this section and section \ref{doublecone}, we only require that inequality \eqref{reg} and inequality \eqref{Eq2.1} hold true for $k$. Since these inequalities are still valid when we replace $k$ and $k - a_n$, the periodicity of betti numbers of $\bar I(\a +j)$ happens when $j > N - a_n$. 
\end{Remark}
As a corollary, we have
\begin{Corollary} If $j > N$, then
$$\reg \bar I(\a+j+b_1) = \reg \bar I(\a+j) + e.$$
In particular, $\reg \bar I(\a + j)$ is quasi-linear in $j$ when $j > N$.
\end{Corollary}
\begin{proof}By Theorem \ref{Per}, it suffices to show that there is at least one minimal binomial generator of $\bar I(\a+j)$ involving $x_0$. This is always the case, since $I(\a+j)$ always contains at least one inhomogeneous minimal generator.
\end{proof}

%%%%%%%%%%%%%%%%%%%%%%%%5
%%%%%%%%%%%%%%%%%%%%%%%%%55
%%%%%%%%%%%%%%%%%%%%%%%%%%
\section{Periodicity of betti numbers of affine monomial curves}\label{general}
The main result of this section is Theorem \ref{Inh}, where we prove that the total betti numbers of $I(\a+j)$ are equal to those of $\bar I(\a + j)$ when $j > N$. From that and the eventual periodicity of betti numbers of $\bar I(\a+j)$ in section \ref{proj}, we prove the  main theorem.

Fix $j > N$. We simply denote $\Delta_{l,r}(j)$ by $\Delta_{l,r}$. Denote by $(\a+j)$ the semigroup generated by $k-b_1, ..., k-b_{n-1}, k$. The ideal $I(\a+j)$ is the defining ideal of the semigroup ring $K[(\a+j)]$. In this setting, Definition \ref{divisorcomplex} gives 

\begin{Definition}\label{deltam} For each $m \in (\a + j)$, let $\Delta_m$ be the simplicial complex on the vertices $\{1,...,n\}$ such that $F \subseteq \{1,...,n\}$ is a face of $\Delta_m$ if and only if the equation 
\begin{equation}\label{Eq4.2}
\sum_{i=1}^n y_i (k-b_i) = m
\end{equation}
has a non-negative integer solution $y = (y_1,..., y_n)$ such that $F \subseteq \supp y =\{i: y_i > 0\}$.
\end{Definition}

In considering the betti numbers of $I(\a+j)$, we will use the following grading coming from the semigroup $(\a + j)$. 
\begin{Definition}[$(\a+j)$-grading] The $(\a+j)$-grading on $R= K[x_1,...,x_n]$ is given by $\deg x_i = k - b_i$.
\end{Definition}
When $R$ is endowed with $(\a+j)$-grading, Theorem \ref{trans} gives
\begin{Proposition}\label{bettij} For each $i$ and each $m$, we have
$$\b_{im}(I(\a+j)) = \dim_K \Tor^R_i(I(\a+j),K)_m = \dim_K \tilde H_i (\Delta_m),$$
where $\Tor^R_i(I(\a+j),K)_m$ is the $(\a+j)$-degree $m$ part of $\Tor^R_i(I(\a+j),K).$
\end{Proposition}

Note that for each pair $(l,r)$ corresponding to an element of $\overline{\a+j}$, $lk-r$ is an element of $(\a+j)$. To prove the equality of total betti numbers of $I(\a+j)$ and of $\bar I(\a +j)$, we prove the equality of homology groups of $\Delta_{lk-r}$ and of $\Delta_{l,r}$. More precisely, applying the double cone structure, we will prove that for each $l > n + \reg J(\a)$, if $\Delta_{l,r}$ has non-trivial homology groups then $\Delta_{lk-r}$ is obtained from $\Delta_{l,r}$ by deleting the vertex $0$. The double cone structure again applies to prove that $\Delta_{l,r}$ and $\Delta_m$ have the same homology groups.

Similar to Lemma \ref{separation}, we first establish the separation in the betti table of $I(\a+j)$. 
\begin{Lemma}\label{inhsep} Assume that $j > N$. Any minimal binomial inhogeneous generator of $I(\a + j)$ has $(\a+j)$-degree larger than $k(n + \reg J(\a))$. In particular, any syzygy of $I(\a + j)$ of $(\a+j)$-degree at most $k(n + \reg J(\a))$ is a syzygy of $J(\a)$. 
\end{Lemma}
\begin{proof} By Theorem \ref{Per} and Remark \ref{Rem}, each inhomogeneous generator of $I(\a + j)$ has degree at least $n + \reg J(\a)  + e$. Thus its $(\a+j)$-degree is at least 
$$(n + e + \reg J(\a)) (k-b_1) > k(n + \reg J(\a))$$
since $k = a_n + j > b_1 + N$, and by the choice of $N$ in \eqref{Eq1}, 
$$e(k-b_1) \ge eN > b_1 (n + \reg J(\a)).$$
The second statement follows immediately.
\end{proof}

The following technical lemma says that for each pair $(l,r)$ for which $l > n + \reg J(\a)$ and $\Delta_{l,r}$ has non trivial homology groups, the corresponding simplicial complex $\Delta_{lk - r}$ is obtained from $\Delta_{l,r}$ by the deletion of the vertex $0$. From Lemma \ref{Lem2.2} and Lemma \ref{Lem2.3}, we see that $\Delta_{l,r}$ and $\Delta_m$ have the same homology groups.

\begin{Lemma}\label{Lem4.1} Assume that $l > n + \reg J(\a)$ and $\Delta_{l,r}$ has non-trivial homology groups. Let $m = lk - r$. If $(y_1, ..., y_n)$ is a non-negative integer solution of \eqref{Eq4.2} then $\sum_{i=1}^n y_i = l -e$ or $\sum_{i=1}^n y_i  = l$. In particular, $\Delta_m$ is obtained from $\Delta_{l,r}$ by the deletion of the vertex $0$.
\end{Lemma} 
\begin{proof} Note that $\sum_{i=1}^n y_i(k - b_i) = lk -r$ is equivalent to 
$$\left ( \sum_{i=1}^n y_i -l\right ) k = \sum_{i=1}^n y_i b_i - r.$$
Since the right hand side is divisible by $d$, the left hand side is divisible by $d$. Therefore $\sum_{i=1}^n y_i - l$ is divisible by $e$. Thus it suffices to show that $\sum_{i=1}^n y_i \ge l-e$ and $\sum_{i=1}^n y_i < l+e.$ 

By Lemma \ref{Lem2.3}, and inequality \eqref{Eq2.1}, 
$$\sum_{i=1}^n y_i \ge \frac{lk -r}{k}\ge \frac{lk - ek - B - dc}{k} > l-e-1.$$
 
Moreover, 
$$\sum_{i=1}^n y_i \le \frac{lk - r}{k-b_1}.$$
To prove that $\frac{lk-r}{k-b_1} < l+e$ is equivalent to prove that
$$lb_1 + eb_1 < r + ek.$$
This follows from Lemma \ref{Lem3.1}, and inequality \eqref{Eq2.1} since
$$lb_1 + eb_1 < r + \left ( \frac{dc+b_1}{b_{n-1}} + n+e \right ) b_1 < r + \frac{k}{b_1b_2}b_1 \le r + ek.$$
Therefore $\sum_{i=1}^n y_i$ is either $l-e$ or $l$. 

Let $\Delta$ be the simplicial complex obtained by deleting the vertex $0$ of the simplicial complex $\Delta_{l,r}$. From Definition \ref{deltalr} and Definition \ref{deltam}, $\Delta \subseteq \Delta_m$. It suffices to show that $\Delta_m \subseteq \Delta.$ Let $F$ be any facet of $\Delta_m$. By Definition \ref{deltam}, there exists a solution $(y_1, ..., y_n)$ of the equation $\sum_{i=1}^n y_i(k-b_i) = m$ such that $\supp y = F$. We have two cases:

If $\sum_{i=1}^n y_i = l$, then $y$ is also a solution of the equation $\sum_{i=1}^n y_i b_i = r$. By Definition \ref{deltalr}, $F$ is a face of $\Delta_{l,r}$ which is also a face of $\Delta$. 

If $\sum_{i=1}^n y_i = l-e$, then $z = (e,y_1,...,y_n)$ is a solution of the equation $y_0 k +\sum_{i=1}^n y_ib_i = r$. By Definition \ref{deltalr}, $F \cup \{0\}$ is a face of $\Delta_{l,r}$, thus $F$ is a face of $\Delta$.
\end{proof}

\begin{Lemma}\label{Lem4.2} Assume that $l_1, l_2 >n+\reg J(\a)$ and $\Delta_{l_1,r_1}$, $\Delta_{l_2,r_2}$ have non-trivial homology groups. Then $l_1k - r_1 \neq l_2k - r_2.$
\end{Lemma}
\begin{proof} Assume that $l_1k-r_1 = l_2k-r_2$. It follows that $k|r_1 - r_2$. Moreover, by Lemma \ref{Lem2.1}, we have
$$ek \le r_1, r_2 \le ek + dc +B.$$
Together with inequality \eqref{Eq2.1} this implies
$$|r_1-r_2| \le dc +B < k$$
Thus $r_1 = r_2$ and so $l_1 = l_2$. 
\end{proof}

\begin{Theorem}\label{Inh} For each $i$, $\b_i(\bar I(\a+j)) = \b_i(I(\a+j))$. 
\end{Theorem}
\begin{proof} Since $I(\a +j)$ is the dehomogenization of $\bar I(\a + j)$, $\b_i(\bar I(\a+j)) \ge \b_i(I(\a+j))$. 

Moreover, by double cone structure in section \ref{doublecone} and Lemma \ref{Lem4.1}, we have if $l > n + \reg J(\a)$ and $\Delta_{l,r}$ has non-trivial homology groups then $\Delta_{l,r}$ and $\Delta_{lk-r}$ have isomorphic homology groups. Together with Proposition \ref{bettij} and Proposition \ref{barj}, we have for each $i$, 
\begin{align*}
\b_i(I(\a + j)) & = \sum_{m \ge 0} \dim_K \tilde H_i (\Delta_m)\\
& = \sum_{m \le k (n+ \reg J(\a))} \dim_K \tilde H_i (\Delta_m) + \sum_{m > k (n + \reg J(\a))} \dim_K \tilde H_i (\Delta_m)\\
& \ge \sum_{l\le n + \reg J(\a)} \dim_K \tilde H_i (\Delta_{l,r}) + \sum_{l > n + \reg J(\a)} \dim_K \tilde H_i (\Delta_{l,r})\\
& = \sum_{l,r} \dim_K \tilde H_i (\Delta_{l,r}) = \b_i(\bar I(\a + j))
\end{align*}
Therefore, $\b_i(\bar I(\a + j)) = \b_i(I(\a + j)).$
\end{proof}

\begin{proof}[Proof of Theorem \ref{main}] Fix $j > N$. By Theorem \ref{Inh}, for each $i$, $\b_i(I(\a+j)) = \b_i(\bar I(\a+j))$. By Theorem \ref{Per}, $\b_i(\bar I(\a+j)) = \b_i(\bar I(\a+j+b_1))$. Thus the betti numbers of $I(\a+j)$ are equal to the corresponding betti numbers of $I(\a+j+b_1)$. 
\end{proof}

Finally, note that in general, the period $b_1$ of the eventual periodicity of betti numbers of $I(\a + j)$ is sharp. In the case $\a$ is an arithmetic sequence it was proven by Gimenez, Sengupta and Srivivasan in \cite{GSS}. In the following, we will show that the period $b_1$ of the eventual periodicity of betti numbers of $I(\a+j)$ is sharp in the case of Bresinsky's sequences. Recall from \cite{B} that, for each $h$, let $\a^h = ((2h-1)2h,(2h-1)(2h+1),2h(2h+1),2h(2h+1)+2h-1)$ be a Bresinsky sequence. Since the minimal homogeneous generators of $I(\a^h +j)$ are the same when $j \gg 0$, it suffices to consider the number of minimal inhomogeneous generators of $I(\a^h + j)$ when $j \gg 0$. For an ideal $I$, we denote by $\mu'(I)$ the number of minimal inhomogeneous generators of $I$. 

Fix $h \ge 2$. We simply denote $\a^h$ by $\a$. In this case, we have $b_1 = 6h-1$, $b_2 = 4h$, $b_3 = 2h-1$ and $B= 12h+3$. Note that $R = K[x_1, x_2,x_3,x_4]$. 

\begin{Lemma}\label{B2} If $j \ge 4b_1b_2(b_2+1)$ then $\mu'(I(\a+j)) = \mu'(I(\a+j+b_1)) \le 6h+1$.
\end{Lemma}
\begin{proof}
We first compute the number $N$ in equation \eqref{Eq1} in this situation. By \cite{B}, 
$$J(\a) = (x_2x_3-x_1x_4,x_3^{4h}-x_2^{2h-1}x_4^{2h+1}, x_1x_3^{4h-1},x_2^{2h}x_4^{2h},...,x_1^{2h+1}x_3^{2h-1}-x_2^{4h}).$$
By Buchberger's algorithm, \cite[Theorem ~15.8]{E}, these elements form a Gr\"obner basis for $J(\a)$ with respect to grevlex order. Thus the initial ideal of $J(\a)$ is 
$$\operatorname{in} (J(\a)) = (x_2x_3, x_3^{4h},x_1x_3^{4h-1},...,x_1^{2h}x_3^{2h},x_2^{4h}).$$
Since 
$$(x_2x_3):x_3^{4h} = (x_2),$$
$$(x_2x_3, x_3^{4h}, ..., x_1^ix_3^{4h-i}):x_1^{i+1}x_3^{4h-i-1} = (x_2,x_3)$$
for all $i = 0, ..., 2h-1$, and 
$$(x_2x_3, x_3^{4h}, ..., x_1^{2h}x_3^{2h}):x_2^{4h} = (x_3),$$
$\operatorname{in}(J(\a))$ has linear quotient. By \cite{HH}, $\reg \operatorname{in}(J(\a)) = 4h$. Moreover, by \cite[Theorem~15.17]{E}, 
$$\reg J(\a) \le \reg \operatorname{in}(J(\a)) = 4h;$$
therefore, $\reg J(\a) = 4h$. 

Note that the conductor of the numerical semigroup generated by $b_1, b_2, b_3$ is $c = 4h(2h-1) - 4h - (2h-1)+1$. Therefore, 
$$N = \max \{b_1 (4 + \reg J(\a)), b_1b_2(\frac{c+b_1}{b_3} + B) \} < 4b_1b_2(b_2+1).$$
By Theorem \ref{main} and the fact that minimal homogeneous generators of $I(\a+j)$ and $I(\a+j+b_1)$ are the same, $\mu'(I(\a+j))  = \mu'(\a+j+b_1)$ when $j \ge 4b_1b_2(b_2+1).$

Fix $j>4b_1b_2(b_2+1)$. We simply denote by $\Delta_{l,r}$ the simplicial complexes associated to elements of the semigroup $\overline{\a + j}.$ By Corollary \ref{gens}, and the fact that $x_1x_4 - x_2x_3\in J(\a)$, any minimal binomial inhomogeneous generator of $I(\a + j)$ is of the form
$$x_1^{u_1} x_2^{u_2} -x_3^{u_3} x_4^{u_1+u_2-1-u_3}, \text{ or } x_1^{v_1} x_3^{v_3} -x_2^{v_2} x_4^{v_1+v_3-1-v_2}.$$

Moreover, for each $u_3$, and each $v_2$, there can be at most one minimal binomial generator of $I(\a+j)$ of the two forms above. By Theorem \ref{Inh} and Definition \ref{deltalr}, these minimal binomial inhomgeneous generators correpond to $\Delta_{l,r}$ where $r$ is of the form $a_4+j + 4h v_2$ or $a_4+j + (2h-1)u_3$. 

Assume that either $v_2 \ge 2h$ or $u_3 \ge 4h+1$, then $r - (a_4 + j + 4h + 2h-1) \ge c$. Using representation as in \eqref{rep}, we can write
\begin{equation}\label{ELB}
r = a_4+j + tb_1 + b_2 + b_3 + w_2 b_2 + w_3 b_3
\end{equation}
for some non-negative integer $w_2, w_3$ such that $c \le w_2b_2 + w_3 b_3 < c  + b_1$. We have 
$$3 + t + w_2 + w_3 \le \frac{r - a_4-j}{b_1} + \frac{c+b_1}{b_3} + 3 < \frac{r}{b_1} \le l$$
since $b_1/b_3 < 4$ and $j > b_1(b_2 + 7)$. By Definition \ref{deltalr}, $\{0,2,3,4\}$ is a face of $\Delta_{l,r}$. Moreover, $b_1 = b_2 + b_3$, thus equation \eqref{ELB} and Definition \ref{deltalr} gives $\{0,1\}$ is a face of $\Delta_{l,r}$. Thus $\Delta_{l,r}$ is connected, so $\Delta_{l,r}$ does not support any minimal generator of $\bar I(\a+j)$. Thus $u_3 \le 4h$ and $v_2 \le 2h-1$. 
\end{proof}

We keep notation as in the proof of Lemma \ref{B2}. For each $u_3$ and each $v_2$, the following lemma gives the explicit form of minimal inhomogeneous generators of $I(\a + j)$.

\begin{Lemma}\label{B1} Let $j = (6h-1)m + s$ for some $m$ and $s$ such that $0\le s \le 6h-2$. Let $s = (2h-1)a - 4hb$ be the unique representation of $s$ in term of $2h-1$ and $4h$ such that $0\le a \le 4h$ and $0\le b \le 2h-1$ and $(a,b)\neq (4h,2h-1)$. If $j > 4b_1b_2(b_2+1)$ then the minimal inhomogeneous generators of $I(\a +j)$ are among the following forms
\begin{align*}
f^2_{v_2} & = x_1^{v_2 + m + 1 - b} x_3^{a+b + 2h - v_2} - x_2^{v_2} x_4^{m+2h+a - v_2}\\
g^2_{v_2} & = x_1^{v_2 + m + 2h - b} x_3^{a+b + 1 - 4h - v_2} - x_2^{v_2} x_4^{m-2h+a - v_2}\\
f^3_{u_3} & = x_1^{m-6h+a+1+u_3}x_2^{10h - 2 - a -b-u_3} - x_3^{u_3}x_4^{m+4h-2-b-u_3}\\
g^3_{u_3} & = x_1^{m-2h+1+a+u_3}x_2^{4h-1-a-b-u_3}- x_3^{u_3}x_4^{m+2h-1-b-u_3}
\end{align*}
where $0\le v_2 \le 2h-1$, and $0\le u_3 \le 4h$.
\end{Lemma}
\begin{proof}
Assume that $f = x_1^{v_1} x_3^{v_3} -x_2^{v_2} x_4^{v_1+v_3-1-v_2}$ is a minimal generator of $I(\a + j)$. By Theorem \ref{Inh} and Definition \ref{deltam}, 
$$(6h- 1)m + s + 2h(2h + 1) + 2h -1 + 4hv_2 = (6h -1)v_1 + (2h - 1)v_3.$$
Equivalently, 
\begin{equation}\label{EL4.4}
s = (2h- 1)(v_1 + v_3 - (m + 2h + 1)) + 4h(v_1 - v_2 - m -1).
\end{equation}
If $f$ is a minimal generator of $I(\a + j)$ then $v_1 +v_3$ is as small as possible so that the equation \eqref{EL4.4} has non-negative integer solutions in $v_1$ and $v_3$. Moreover, by Lemma \ref{Lem2.1}, $v_1 + v_3 - (m+2h+1) > -(2h+1).$ Therefore, either 
$$v_1 + v_3 - (m+2h+1) = a, \text{ and } v_1 - v_2 - m-1 = -b$$
or 
$$v_1 + v_3 - (m+2h+1) = -(2h-1-b), \text { and } v_1 - v_2 - m-1 = 4h-a.$$
The first case gives the family $f^2_{v_2}$, while the second case gives the family $g^2_{v_2}$.

Assume that $g = x_1^{u_1} x_2^{u_2} -x_3^{u_3} x_4^{u_1+u_2-1-u_3}$ is a minimal generator of $I(\a + j)$. By Theorem \ref{Inh} and Definition \ref{deltam}, 
$$(6h-1)m + s + 2h(2h+1) + 2h-1 + (2h-1)u_3 = (6h-1)u_1 + 4hu_2.$$
Equivalently, 
\begin{equation}\label{EL4.5}
s = (2h-1)(u_1 -u_3-(m-2h+1)) + 4h(u_1 +u_2 -m- 2h).
\end{equation}
If $g$ is a minimal generator of $I(\a+j)$ then $u_1 + u_2$ is as small as possible so that the equation \eqref{EL4.5} has non-negative integer solutions in $u_1$ and $u_2$. Moreover, by Lemma \ref{Lem2.1}, $u_1 + u_2 -(m+2h) > - 2h$. Therefore, either 
$$u_1 - u_3 - (m-2h+1) = a, \text{ and } u_1 + u_2 -m-2h = -b$$
or 
$$u_1 - u_3 -(m-2h+1) = -(2h-1-b), \text{ and } u_1 +u_2 -m-2h = 4h-a.$$
The first case gives the family $g^3_{u_3}$, while the second case gives the family $f^3_{u_3}$.
\end{proof}

\begin{Proposition}\label{Bre} If $j > 4b_1b_2(b_2+1)$ then $\mu'(I(\a^h + j) ) \le  6h-1$. Moreover, equality happens if and only if $j = 4h \mod 6h-1$. In particular, the period of the periodicity of the betti numbers of $I(\a^h + j)$ in $j$ is exactly $6h-1$.
\end{Proposition}
\begin{proof} We keep the notation as in Lemma \ref{B1}. We have the following cases. 

If $a + b < 4h-1$, then $g^2_i \notin R$ for any $i\ge 0$. Moreover, $f^2_{0} - x_4^{a+b+1} g^3_0 \in J(\a)$, thus $f^2_0$ is not minimal. Also, $f^2_{2h-1} - x_3^{a+b+1}g^3_{4h-1-a-b} \in J(\a)$, thus $f^2_{2h-1}$ is not minimal. Finally, note that $g_i^3\notin R$ for $i > 4h-1-a-b$, and $f^3_{4h-a-b} -x_3x_4^{2h-2}g^3_{4h-1-a-b} \in J(\a)$ which is not minimal. Thus $\mu'(I(\a+j)) \le 6h-2$.

If $a+b = 4h-1$, then $g^2_0= g^3_0$ is the only element in the family $g^2$ belongs to $I(\a+j)$ and $f^2_{v_2} -x_2^{v_2} x_4^{4h-v_2}g^2_0 \in J(\a)$, thus no element in the family $f^2$ are minimal. Thus $\mu'(I(\a+j)) \le 4h + 1$. 

If $4h-1 < a + b \le 6h - 3$, then $g^3_i\notin R$ for any $i \ge 0$. Moreover, $f^3_0 -x_4^{6h-2 - (a+b)} g^2_0 \in J(\a)$, thus $f^3_0$ is not minimal. Also, $f^3_{4h} - x_2^{6h-2-(a+b)}g^2_{a+b+1-4h} \in J(\a)$, thus $f^3_{4h}$ is not minimal. Finally, note that $g^2_i\notin R$ for $i > a+b+1-4h$, and $f^2_{a+b+2-4h}-x_2x_4^{4h-1}g_{a+b+1-4h} \in J(\a)$ which is not minimal. Thus $\mu'(I(\a+j)) \le 6h-2$. 

Finally, if $a + b = 6h-2$ then $a = 4h$ and $b = 2h-2$ and then $s = 4h$. In this case, the minimal inhomogeneous generators of $I(\a + j)$ are among
$$x_1^{m+2h+1-i}x_2^{i} - x_3^{4h - i}x_4^{m-2h+i}$$
for $i  = 0, ..., 4k-1$, and 
$$x_1^{m+2h+1-i}x_3^i -x_2^{2h-1-i} x_4^{m+i+1}$$
for $i = 1, ..., 2k-1$. 

Moreover, each of these generators has degree exactly $m+2h+1$. For each $r$ of the form $r = a_4+j + 4hv_2$ or $r = a_4 + j+(2h-1)u_3$ where $0\le v_2\le 4h-1$ and $0\le u_3 \le 2h-2$, it is easy to check that $\Delta_{m+2h+1,r}$ is disconnected. By Theorem \ref{Inh}, $\mu'(I(\a + j)) = 6h - 1$.
\end{proof}

\section*{Acknowledgements}
I would like to thank J\"urgen Herzog and Hema Srinivasan for introducing the problem and having lots of inspring conversations. Also, I would like to thank my advisor David Eisenbud for useful conversations and comments on earlier drafts of the paper. Finally, I would like to thank Dan Grayson and Mike Stillman for making Macaulay2 \cite{GS}, for which all the computations in this paper are carried out.

%%%%%%%%%%%%
%%%%%%%%%%%%
%%%%%%%%%%%%

\end{document}